\documentclass[11pt]{amsart}
\usepackage[margin=1.35in]{geometry}
\usepackage{amscd,amsmath,amsxtra,amsthm,amssymb,stmaryrd,xr,mathrsfs,mathtools,enumerate,commath, comment}
\usepackage{stmaryrd}
\usepackage{xcolor}
\usepackage{commath}
\usepackage{comment}
\usepackage{tikz-cd}
\usepackage{longtable} 
\usepackage{pdflscape} 
\usepackage{booktabs}
\usepackage{hyperref}
\definecolor{vegasgold}{rgb}{0.77, 0.7, 0.35}
\definecolor{darkgoldenrod}{rgb}{0.72, 0.53, 0.04}
\definecolor{gold(metallic)}{rgb}{0.83, 0.69, 0.22}
\hypersetup{
 colorlinks=true,
 linkcolor=darkgoldenrod,
 filecolor=brown,      
 urlcolor=gold(metallic),
 citecolor=darkgoldenrod,
 pdftitle={Hilbert's tenth problem in Anticyclotomic towers of number fields},
 }

\usepackage[all,cmtip]{xy}

\DeclareFontFamily{U}{wncy}{}
\DeclareFontShape{U}{wncy}{m}{n}{<->wncyr10}{}
\DeclareSymbolFont{mcy}{U}{wncy}{m}{n}
\DeclareMathSymbol{\Sh}{\mathord}{mcy}{"58}
\usepackage[T2A,T1]{fontenc}
\usepackage[OT2,T1]{fontenc}

\newtheorem{theorem}{Theorem}[section]
\newtheorem{lemma}[theorem]{Lemma}
\newtheorem{ass}[theorem]{Assumption}
\newtheorem*{theorem*}{Theorem}
\newtheorem*{ass*}{Assumption}
\newtheorem{definition}[theorem]{Definition}
\newtheorem{corollary}[theorem]{Corollary}

\newtheorem{conjecture}[theorem]{Conjecture}
\newtheorem{proposition}[theorem]{Proposition}

\newcommand{\cF}{\mathcal{F}}

\newcommand{\Zp}{\mathbb{Z}_p}

\newcommand{\Kant}{K_{\operatorname{anti}}}

\newcommand{\cH}{\mathcal{H}}

\newcommand{\Z}{\mathbb{Z}}
\newcommand{\p}{\mathfrak{p}}
\newcommand{\Q}{\mathbb{Q}}
\newcommand{\F}{\mathbb{F}}

\newcommand{\cO}{\mathcal{O}}

\newcommand{\corank}{\mathrm{corank}}

\newcommand{\op}[1]{\operatorname{#1}}

\numberwithin{equation}{section}

\begin{document}

\title[Hilbert's tenth problem in Anticyclotomic towers]{Hilbert's tenth problem in Anticyclotomic towers of number fields}

\author[A.~Ray]{Anwesh Ray}
\address[Ray]{Chennai Mathematical Institute, H1, SIPCOT IT Park, Kelambakkam, Siruseri, Tamil Nadu 603103, India}
\email{ar2222@cornell.edu}

\author[T.~Weston]{Tom Weston}
\address[Weston]{Department of Mathematics, University of Massachusetts, Amherst, MA, USA.} 
  \email{weston@math.umass.edu}

\keywords{Anticyclotomic Iwasawa theory, Selmer groups, Hilbert's tenth problem, congruences between elliptic curves}
\subjclass[2020]{11R23, 11U05}

\maketitle

\begin{abstract}
Let $K$ be an imaginary quadratic field and $p$ be an odd prime which splits in $K$. Let $E_1$ and $E_2$ be elliptic curves over $K$ such that the $\op{Gal}(\bar{K}/K)$-modules $E_1[p]$ and $E_2[p]$ are isomorphic. We show that under certain explicit additional conditions on $E_1$ and $E_2$, the anticyclotomic $\Z_p$-extension $\Kant$ of $K$ is integrally diophantine over $K$. When such conditions are satisfied, we deduce new cases of Hilbert's tenth problem. In greater detail, the conditions imply that Hilbert's tenth problem is unsolvable for all number fields that are contained in $\Kant$. We illustrate our results by constructing an explicit example for $p=3$ and $K=\Q(\sqrt{-5})$.  
\end{abstract}

\section{Introduction}
 \par Hilbert asked whether there exists a deterministic algorithm for determining the existence of integral solutions to diophantine equations with integer coefficients. In other words, the problem asks whether there exists a Turing machine which takes as input a finite set of polynomial equations over $\Z$ and decides whether they have nontrivial solutions in $\Z$. If no such Turing machine does exist, then it is said that Hilbert's tenth problem is unsolvable. For further details, we refer to \cite{shlapentokh2007hilbert}. A celebrated result due to Matiyasevich \cite{matiyasevich1970diophantineness}, building upon previous work of Davis, Putnam and Robinson \cite{davis1961decision}, shows that Hilbert's tenth problem for diophantine equations with integer coefficients is unsolvable. One of the main open problems in the area is to study the generalization of Hilbert's tenth problem to number rings.  For a number field $L$, generally this is studied by showing that $\Z$ is a 
 \emph{integrally diophantine} subset of $\cO_L$, in the sense of Definition \ref{diophantine subset}.  It is then known that $\cO_L$ inherits the unsolvability of Hilbert's tenth problem from $\Z$.  The following conjecture predicts that Hilbert's tenth problem is unsolvable for all rings of integers of number fields.

\begin{conjecture}[Denef-Lipshitz]\label{conj denef lip}
For any number field $L$, $L/\Q$ is an integrally diophantine extension. 
\end{conjecture}

 Classical results in the field show that Conjecture \ref{conj denef lip} above
is true for the following classes of number fields $L$:
\begin{itemize}
    \item The conjecture holds when $L$ is either totally real or a quadratic extension of a totally real number field (cf. \cite{denef1980diophantine,denef1978diophantine}), 
    \item $L$ has exactly one complex place (cf. \cite{pheidas1988hilbert, shlapentokh1989extension, videla16decimo}),
    \item $L/\Q$ is abelian (cf. \cite{shapiro1989diophantine}),
    \item $[L:\Q]=4$, $L$ is not totally real and $L/\Q$ has a proper intermediate field (cf. \cite{denef1978diophantine}).
\end{itemize}
Poonen \cite{poonen2002using}, Cornelissen-Pheidas-Zahidi \cite{cornelissen2005division} and Shlapentokh \cite{shlapentokh2008elliptic} subsequently established a geometric criterion for Conjecture \ref{conj denef lip} to hold for a given number field $L/\Q$. In establishing some new cases of Conjecture \ref{conj denef lip}, we use the following criterion of Shlapentokh.
\begin{theorem}[Shlapentokh]\label{shlap 1}
Let $L/K$ be an extension of number fields and suppose that there exists an elliptic curve $E_{/K}$ such that $\op{rank}E(L)=\op{rank}E(K)>0$. Then, $L/K$ is integrally diophantine.
\end{theorem}
We mention in passing that this condition has been recently generalized to the context of abelian varieties by Mazur, Rubin and Shlapentokh \cite{MRS2023}. Let $K$ be a number field for which Conjecture \ref{conj denef lip} holds and let $L/K$ be a number field extension. Suppose that there exists an elliptic curve such that $\op{rank}E(L)=\op{rank}E(K)>0$. Then, it follows that Conjecture \ref{conj denef lip} holds for $L$. Theorem \ref{shlap 1} has been applied in various contexts. For instance, Garcia-Fritz and Pasten \cite{garcia2020towards} established that for certain families of primes $\mathscr{P}$ and $\mathscr{Q}$ with positive density, Conjecture \ref{conj denef lip} is satisfied for $L=\Q(p^{1/3}, \sqrt{-q})$ with $p \in \mathscr{P}$ and $q \in \mathscr{Q}$. Mazur and Rubin \cite{mazur2018diophantine} show that if the Conjecture \ref{conj denef lip} holds for a given number field $K$, then it also holds for certain infinite families of degree $p^n$ cyclic extensions $L/K$, where $p$ is a prime and $n\in \Z_{\geq 1}$.

\subsection{Main results}
\par In what follows, $p$ is a fixed odd prime number and $K=\Q(\sqrt{-D})$ is an imaginary quadratic field in which $p$ splits into two primes $\p$ and $\p^*$. Let $\Z_p$ denote the ring of $p$-adic integers. A $\Z_p$-extension of $K$ is an infinite Galois extension $K_\infty$ of $K$ such that $\op{Gal}(K_\infty/K)$ is isomorphic to $\Z_p$ as a topological group. The \emph{anticyclotomic}-extension $\Kant/K$ is the unique $\Z_p$-extension of $K$ which is Galois over $\Q$ such that $\op{Gal}(\Kant/\Q)$ is a non-abelian pro-dihedral group. Let $K_n$ be the $n$-th layer, i.e., the number field such that $K\subseteq K_n\subseteq \Kant$ and $[K:K_n]=p^n$. The group $\op{Gal}(K_n/\Q)$ is isomorphic to the dihedral group with $2p^n$ elements. The tower of fields 
\[K=K_0\subseteq K_1\subseteq \dots \subseteq K_n \subseteq K_{n+1}\subseteq \dots\] is referred to as the \emph{anticyclotomic $\Z_p$-tower} over $K$. The fields $K_n$ are totally imaginary non-abelian extensions of $\Q$, containing $p^n$ complex places. Conjecture \ref{conj denef lip} is not known to hold for the fields $K_n$, when $n>0$. We do however note that $K/\Q$ is known to be integrally diophantine (by the results above). If there exists an elliptic curve $E_{/K}$ such that $\op{rank}E(K_n)=\op{rank} E(K)>0$, then, it follows from Slapentokh's criterion (i.e., Theorem \ref{shlap 1}) that $K_n/K$ is integrally diophantine. Since $K/\Q$ is integrally diophantine, it shall then follow that Conjecture \ref{conj denef lip} is satisfied for $K_n/\Q$. We say that $\Kant/K$ is integrally diophantine if $K_n/K$ is integrally diophantine for all $n$. We note that if $\Kant/K$ is integrally diophantine, it follows that Conjecture \ref{conj denef lip} holds for $K_n$ for all $n$. We mention here that ideas from Iwasawa theory have been used to prove new cases of Hilbert's tenth problem in certain cyclotomic towers. This has been done in \cite[Example 7.1]{shlapentokh2008elliptic}. For related observations, cf. \cite{Ray2023HTP}.
 \par In order to study the stability of ranks of elliptic curves in anticyclotomic towers, we use techniques from the Iwasawa theory of Selmer groups. In order to state our main results, we introduce some further notation. Let $E_1$ and $E_2$ be elliptic curves defined over $K$. Fix an algebraic closure $\bar{K}$ of $K$, and set $E_i[p]$ to denote the $p$-torsion subgroup of $E_i(\bar{K})$. We assume that $E_1[p]$ and $E_2[p]$ are isomorphic as modules over $\op{Gal}(\bar{K}/K)$. Such a pair of elliptic curves $(E_1,E_2)$ shall be referred to as a \emph{$p$-congruent} pair of elliptic curves over $K$. Let $\mathfrak{T}$ be the set of primes $v$ of $K$ such that either $E_1$ or $E_2$ (or both) have bad reduction at $v$. Let $\mu_p\subset \bar{K}$ denote the $p$-th roots of unity and set $K':=K(\mu_p)$, i.e., $K'$ is the number field generated by $K$ and $\mu_p$. At a finite prime $v$ of $K$, denote by $\F_v$ the residue field of the ring of integers at $v$. If $E_{/K}$ is an elliptic curve with good reduction at $v$, we set $E(\F_v)$ to be denote the group of $\F_v$-points of a regular model of $E$. For $i=1,2$, we define a subset $\Sigma(E_i)$ of $\mathfrak{T}$ below, which will play a role in the statement of our results.

\begin{definition}
For $i=1,2$, the set $\Sigma(E_i)$ consists of primes $v\in \mathfrak{T}$ such that the following conditions is satisfied.
\begin{enumerate}
\item If the elliptic curve $E_i$ has good reduction at $v$, then, $p\mid \#E_i(\F_v)$.
\item If the elliptic curve $E_i$ has bad reduction at $v$, and $p\geq 5$, then $E_i$ has split multiplicative reduction at some prime $w|v$ of $K'$.
\item If the elliptic curve $E_i$ has bad reduction at $v$, and $p=3$, then, at some prime $w|v$ of $K'$, the elliptic curve $E_i$ has either split multiplicative reduction or has additive reduction of type $\rm{IV}$ or $\rm{IV}^*$. 
\end{enumerate}
\end{definition}
When referring to a prime $v$ of $K$, it shall be understood, unless stated otherwise, that one is referring to a finite prime. Let $v$ be a prime of $K$ and let $\ell$ be the prime number for which $v|\ell$. The prime $v$ is finitely decomposed in $\Kant$ if and only if $\ell$ splits in $K$ into two distinct conjugate primes $v$ and $v^*$.  If $v$ is finitely decomposed in $\Kant$ (i.e., $\ell$ splits in $K$), then we set $s_v$ to denote the number of primes of $\Kant$ that lie above $v$.  Clearly $s_v = s_{v^*}$.  We note that \cite[Theorem 2, p. 2133]{brink2007prime} gives an explicit algorithm to compute the number $s_v$. For an elliptic curve $E_{/K}$, set $V_p(E):=\op{T}_p(E)\otimes_{\Z_p} \Q_p$, where $T_p(E)$ is the $p$-adic Tate module associated to $E$. Let $\sigma_v$ denote the Frobenius in $\op{Gal}(K_v^{\op{nr}}/K_v)$ and set $\op{I}_v$ to denote the inertia group at $v$. Let $V_p':=(V_p)_{\op{I}_v}$ be the maximal quotient of $V_p$ on which $\op{I}_v$ acts trivially. Let $P_v(X)=P_v^{E}(X)\in \Z_p[X]$ denote the characteristic polynomial $\op{det}\left(1-\sigma_{v|V_p'}X\right)$. When $E$ has good reduction at $v$, then, $P_v(X)$ is given by $P_v(X)=1-a_v(E) X+q_v X^2$, where, $q_v$ is the number of elements in $\F_v$, and $a_v(E):=q_v+1-\#E(\F_v)$. When $E$ has multiplicative reduction at $v$, then, $P_v(X)=1-a_v(E) X$, where $a_v(E)=+1$ (resp. $-1$) if $E$ has split (resp. non-split) multiplicative reduction. When $E$ has additive reduction at $v$, $P_v(X)=1$. Let $\widetilde{P}_v(X)=\widetilde{P}_v^E(X)\in \Z/p\Z[X]$ be the mod-$p$ reduction of $P_v(X)$. Let $d_v(E)$ be the multiplicity of $\ell^{-1}$ as a root of $\widetilde{P}_v(X)$.  Note that if $E$ is defined over $\Q$, then, $P_v(X)=P_{v^*}(X)$ and $d_v(E)=d_{v^*}(E)$. Denote by $E_0(K_v)$ the subgroup of points in $E(K_v)$ of good reduction. The \emph{Tamagawa number} $c_v(E)$ is the index $[E(K_v):E_0(K_v)]$. Let $|\cdot|_p$ be the nonarchimedian absolute value normalized so that $|p|_p=p^{-1}$. The $p$-primary part of $c_v(E)$ is defined as $c_v^{(p)}(E):=|c_v(E)|_p^{-1}$. Set $\Omega_K$ to be the set of primes $v\nmid p$ of $K$ which are finitely decomposed in $\Kant$. Set $\Sh(E/K)[p^\infty]$ to denote the $p$-primary part of the Tate-Shafarevich group of $E$.

\begin{theorem}\label{main theorem}
Let $K$ be an imaginary quadratic field and let $p$ be an odd prime number which splits into primes $\p$ and $\p^*$ in $K$. Let $\Kant$ be the anticyclotomic $\Z_p$-extension of $K$, and for $n\in \Z_{\geq 0}$ denote by $K_n$, its $n$-th layer. With respect to notation above, assume that there exist elliptic curves $E_1$ and $E_2$ defined over $K$ such that 
\begin{enumerate}
\item both $E_1$ and $E_2$ have good ordinary reduction at $\p$ and $\p^*$,
\item $E_1[p]$ and $E_2[p]$ are isomorphic as modules over $\op{Gal}(\bar{K}/K)$,
\item $\op{rank} E_1(K)=0$ and $\op{rank}E_2(K)>0$, 
\item $E_1(K)[p]=E_2(K)[p]=0$, 
\item\label{c5 of thm main} each prime $v\in \Sigma(E_1)\cup \Sigma(E_2)$ is finitely decomposed in the anticyclotomic extension $\Kant/K$. Equivalently, let $v\in \Sigma(E_1)\cup \Sigma(E_2)$ and $\ell$ be the prime number for which $v|\ell$. Then, $\ell$ splits in $K$ into a product of two primes $v$ and $v^*$.
\item\label{c6 of thm main} $\Sh(E_1/K)[p^\infty]=0$, 
\item $E_1(\F_{\p})[p]=0$.
\item\label{c8 of thm main} Recall that $\Omega_K$ is the set of primes $v\nmid p$ of $K$ which are finitely decomposed in $\Kant$. Assume that $c_v^{(p)}(E_1)=1$ for all $v\in \Omega_K$ (at which $E_1$ has bad reduction). 
 
\end{enumerate}
Furthermore, assume that the following relationship holds
\begin{equation}\label{main equation for rank}\op{rank} E_2(K)=\sum_{v\in \Sigma(E_1)\cup \Sigma(E_2)} s_v\left(d_v(E_1)-d_v(E_2)\right).\end{equation}

Then, for all $n\in \Z_{\geq 0}$, $\cO_{K_n}$ satisfies Conjecture \ref{conj denef lip}, and thus Hilbert's tenth problem has a negative answer for $\cO_{K_n}$.
\end{theorem}

\par Note that when the elliptic curves $E_1$ and $E_2$ are defined over $\Q$, the condition \eqref{c5 of thm main} above implies that the sum $\sum_{v\in \Sigma(E_1)\cup \Sigma(E_2)} s_v\left(d_v(E_1)-d_v(E_2)\right)$ is even, and therefore, the rank of $E_2(K)$ is even. The condition \eqref{c6 of thm main} in the theorem above implies in particular that $\Sh(E_1/K)[p^\infty]$ is finite. However, we do not need to assume that $\Sh(E_2/K)[p^\infty]$ is finite. We illustrate the above theorem via an explicit example for which all the conditions are satisfied. In what follows, let $K=\Q(\sqrt{-5})$, $p=3$ and consider the pair of rational elliptic curves $E_1=$ \href{https://www.lmfdb.org/EllipticCurve/Q/56/b/2}{56b1} and $E_2=$ \href{https://www.lmfdb.org/EllipticCurve/Q/392/c/1}{392c1} with respect to Cremona's labelling for elliptic curves. Note that both elliptic curves are defined over $\Q$. The residual representations 
\[\bar{\rho}_i: \op{Gal}(\bar{\Q}/\Q)\rightarrow \op{GL}_2(\Z/3\Z)\] are surjective and $\bar{\rho}_1$ is isomorphic to $\bar{\rho}_2$. The ranks of the Mordell Weil groups $E_1(\Q)$ and $E_2(\Q)$ are $0$ and $1$ respectively. One finds that $\op{rank} E_1(K)$ remains $0$, while for $E_2$, there is a rank jump and $\op{rank}E_2(K)=2$. Computations were done on the {\tt SageMathCloud} and the {\tt Magma}  computational algebra system to show that the conditions of Theorem \ref{main theorem} for $(K,p, E_1, E_2)$ are satisfied. We obtain the following result, illustrating Theorem \ref{main theorem}.
\begin{theorem}\label{theorem main aux}
Let $K=\Q(\sqrt{-5})$ and for $n\geq 0$, let $K_n$ be the $n$-th layer in the anticyclotomic $\Z_3$-extension of $K$. The curve $E_1=56b1$ has rank $0$ over $K$. Then, for all $n\in \Z_{\geq 0}$, $\cO_{K_n}$ satisfies Conjecture \ref{conj denef lip}. Thus, this shows that Hilbert's tenth problem has a negative answer for $\cO_{K_n}$.
 \end{theorem}
 
 \subsection{Method of proof}
\par Let us briefly describe the method of proof of the Theorem \ref{main theorem}. Let $E_{/K}$ be an elliptic curve and $p$ be a prime such that $E$ has good ordinary reduction at the primes of $K$ above $p$. Let $\op{Sel}_{p^\infty}(E/\Kant)$ be the $p$-primary Selmer group of $E$ over $\Kant$. We show that if there exists an elliptic curve $E$ for which the following conditions hold
\begin{enumerate}
\item $E$ has good ordinary reduction at the primes $\p$ and $\p^*$,
\item $\op{rank}E(K)>0$,
\item the Selmer group $\op{Sel}_{p^\infty}(E/\Kant)$ is a cofinitely generated $\Z_p$-module, and
\item the $\lambda$-invariant of $\op{Sel}_{p^\infty}(E/\Kant)$ is equal to $\op{rank} E(K)$, 
\end{enumerate}
then, $K_n/\Q$ is integrally diophantine for all $n$. The proof relies on basic properties of the $\lambda$-invariant and Theorem \ref{shlap 1}. We refer to Theorem \ref{first theorem section 3} and its proof for further details. In order to construct an elliptic curve $E_{/K}$ such that the above conditions are satisfied, we exploit the Iwasawa theory of congruences between elliptic curves, initiated by Greenberg-Vatsal \cite{greenberg2000iwasawa} and further refined by Kundu-Ray \cite[section 6]{kunduraykida}. This refinement does not only generalize the results of Greenberg and Vatsal to non-cyclotomic $\Z_p$-extensions, but it gives a more accurate description of the primes that actually contribute to the change in $\lambda$-invariant. This optimal set of primes is the set $\Sigma(E_1)\cup \Sigma(E_2)$, and for the example considered in section \ref{s 5}, this set is strictly smaller than the set of all primes of bad reduction for one or both elliptic curves. Condition \eqref{c5 of thm main} stipulates that each of primes $v\in \Sigma(E_1)\cup \Sigma(E_2)$ must be finitely decomposed in $\Kant$, which is the case for the example considered. It is thus ideal for the set $\Sigma(E_1)\cup \Sigma(E_2)$ for which this assumption is made, to be as small as possible.

\par We consider two elliptic curves $E_1$ and $E_2$ defined over $K$ that are $p$-congruent satisfying the conditions of Theorem \ref{main theorem}. Since $\op{rank} E_1(K)=0$, and the $p$-primary part of $\Sh(E_1/K)$ is finite (in fact, it is assumed to be zero), it follows that the $p$-primary Selmer group of $E_1$ over $K$ is finite. By a well known result of Greenberg that $\op{Sel}_{p^\infty}(E_1/\Kant)$ is a cotorsion module over the Iwasawa algebra (cf. Proposition \ref{prop 2.1}). Furthermore, the assumptions \eqref{c6 of thm main}-\eqref{c8 of thm main} imply that the $\mu$ and $\lambda$-invariants of $\op{Sel}_{p^\infty}(E_1/\Kant)$ are zero (cf. Corollary \ref{cor 4.6}). Such a criterion relies on a formula for the Euler characteristic over the anticyclotomic extension, due to Van Order \cite{van15}. At this step, it is crucial for $E_1$ to have rank $0$ over $K$. We set $E=E_2$, which is an elliptic curve with positive rank over $K$, and show that the above conditions are satisfied. The congruence with the rank $0$ elliptic curve $E_1$, along with the assumptions of Theorem \ref{main theorem} imply that $\op{Sel}_{p^\infty}(E/\Kant)$ is a cofinitely generated $\Z_p$-module satisfying the three conditions above. At this step we make use of Theorem \ref{thm 3.5}, which essentially extends the results of \cite[section 6]{kunduraykida} to $p=3$.

\subsection{Outlook/ future directions} Our results show that Iwasawa theory can be used as a novel tool to prove new interesting cases of Conjecture \ref{conj denef lip}. It is natural to extend such investigations to more general contexts that involve $p$-adic Lie extensions of dimension $d>1$. Let us illustrate this by considering a special case: the $\Z_p^2$-extension $\widetilde{K}_\infty$ of an imaginary quadratic field $K$. For $n\geq 0$, let $\widetilde{K}_n$ denote the $n$-th layer, i.e., the extension of $K$, contained in $\widetilde{K}_\infty$ for which $\op{Gal}(\widetilde{K}_n/K)\simeq (\Z/p^n \Z)^2$. Let $E$ be an elliptic curve over $\Q$ with good ordinary reduction at $p$. When the root number of the Hasse-Weil L-function of $E_{/K}$ is $+1$, then, a result of Cornut and Vatsal \cite{cornut2007nontriviality} shows that the Mordell-Weil rank of $E(\widetilde{K}_n)$ is bounded as $n\rightarrow \infty$. Suppose that $\op{rank}E(K)>0$ and the root number of $E_{/K}$ is equal to $+1$. Following Mazur and Rubin, we say that the Mordell-Weil rank is \emph{diophantine stable} in $\widetilde{K}_\infty$ if \[\op{rank}E(\widetilde{K}_n)=\op{rank}E(K)\] for all $n$. Then, from Theorem \ref{shlap 1}, it is evident that this property implies that Conjecture \ref{conj denef lip} is satisfied for all number fields $\widetilde{K}_n$. Generalized upper bounds for the asymptotic growth of the rank of the Mordell--Weil group of an abelian variety in a tower of number fields contained in a $p$-adic Lie extension, were proven by Harris \cite{harris2000correction}. The precise order of growth in Mordell-Weil ranks is however not well understood \cite[Question following Proposition 4.7]{hung2021growth}. It is our expectation that the search for diophantine stability in the context of Iwasawa theory would lead to interesting theoretical and computational developments.

\subsection{Organization} Including the introduction, the paper consists of five sections. In section \ref{s 2}, we set up notation and discuss preliminary notions from anticyclotomic Iwasawa theory. We also discuss the notion of an integrally diophantine extension in this section. In section \ref{s 3}, we discuss congruences between elliptic curves and their effect on Selmer groups. More specifically, in this section, we recall the results of \cite[section 6]{kunduraykida}. These results apply only to primes $p\geq 5$, and our main task in section \ref{s 3} is to extend those results to the case when $p=3$. In section \ref{s 4}, we study Iwasawa theory in the context of Hilbert's tenth problem, and prove Theorem \ref{main theorem}. In section \ref{s 5}, we illustrate Theorem \ref{main theorem} through an explicit example, and prove Theorem \ref{theorem main aux} as a consequence.

\subsection{Acknowledgment} When the project was started, the author was a Simons postdoctoral fellow at the Centre de recherches mathematiques in Montreal, Canada. At this time, the author's research is supported by the CRM Simons postdoctoral fellowship. He thanks Antonio Lei for some helpful discussions. We thank the referee for the excellent report.

\section{Preliminaries}\label{s 2}
\subsection{Iwasawa theory of elliptic curves}
\par In this section, we review some necessary background. For a more comprehensive introduction to the Iwasawa theory of elliptic curves, the reader is referred to \cite[Chapter 1]{coates2000galois}. Throughout, $p$ is an odd prime number and $E$ shall denote an elliptic curve over a number field $K$. We assume that $E$ has good ordinary reduction at the primes $v|p$. We fix an algebraic closure $\bar{K}$ of $K$. Let $\Z_p$ denote the $p$-adic integers. An infinite Galois extension $K_\infty/K$ is said to be a $\Z_p$-extension if there is a topological isomorphism of groups $\op{Gal}(K_\infty/K)\xrightarrow{\sim} \Z_p$. Let $K_\infty/K$ be a $\Z_p$-extension, and set $\Gamma:=\op{Gal}(K_\infty/K)$. The \emph{Iwasawa algebra} is the completed group ring defined as follows \[\Lambda:=\Z_p\llbracket \Gamma \rrbracket:=\lim_{n\rightarrow \infty} \Z_p[\Gamma/\Gamma^{p^n}].\]  Fixing a topological generator $\gamma$ of $\Gamma$ yields the usual isomorphism $\Lambda \cong \Zp\llbracket T \rrbracket$ with the ring of integral $p$-adic power series. This isomorphism identifies the group-like element $\gamma$ with $(T+1)$.
 For $n\in \Z_{\geq 0}$, set $K_n$ to denote the $n$-th layer of the associated tower. This is the extension $K_n/K$ such that $K_n\subset K_\infty$ and $\op{Gal}(K_n/K)\simeq \Z/p^n\Z$. Thus, there is an infinite tower of number fields
\[K_0\subset K_1\subset \dots \subset K_n\subset K_{n+1}\subset \dots,\] where $K=K_0$. We shall refer to such a tower as a $\Z_p$-tower over $K$. 

\par Let $F/K$ be a number field extension. Given a finite prime $v$ of $K$, set $v(F)$ to denote the set of primes of $F$ that lie above $v$. We set 
\[J_v(E/F):=\bigoplus_{w\in v(F)} H^1(F_w, E)[p^\infty],\] and note that there is a natural map 
\[\iota_v: H^1(\bar{F}/F, E[p^\infty])\longrightarrow J_v(E/F)\] which factors through the restriction maps
\[H^1(\bar{F}/F, E[p^\infty])\longrightarrow \bigoplus_{w\in v(F)} H^1(F_w, E[p^\infty]).\]
The $p$-primary Selmer group of $E/F$ is defined as follows
\[\op{Sel}_{p^\infty}(E/F):=\op{ker}\left(H^1(\bar{F}/F, E[p^\infty])\xrightarrow{\iota} \prod_{v} J_v(E/F)\right),\] where the map $\iota$ is induced by the maps $\iota_v$, and the product ranges over all primes $v$ of $K$.
\par We set $\op{Sel}_{p^\infty}(E/K_\infty)$ to denote the direct limit $\lim_{n\rightarrow \infty} \op{Sel}_{p^\infty}(E/K_n)$, with respect to restriction maps. Let $S$ be a set of primes of $K$ containing the primes above $p$ and the primes at which $E$ has bad reduction. Denote by $K_S$ the maximal algebraic extension of $K$ in which all primes $v\notin S$ are unramified. For $v\in S$, set $J_v(E/K_\infty)$ to be the direct limit $\lim_{n\rightarrow \infty} J_v(E/K_n)$. The Selmer group $\op{Sel}_{p^\infty}(E/K_\infty)$ is described as follows
\begin{equation}\label{selmer group definition alt}\op{Sel}_{p^\infty} (E/K_\infty):=\op{ker}\left(H^1(K_S/K_\infty, E[p^\infty])\longrightarrow \bigoplus_{v\in S} J_v(E/K_\infty)\right).\end{equation}Given a module $M$ over $\Z_p$, we set $M^{\vee}=\op{Hom}_{\Z_p}\left(M, \Q_p/\Z_p\right)$ to denote its \emph{Pontryagin dual}. A module $M$ over $\Lambda$ is said to be cofinitely generated (resp. cotorsion) over $\Lambda$ if $M^\vee$ is finitely generated (resp. torsion) over $\Lambda$. We discuss conditions for $\op{Sel}_{p^\infty}(E/K_\infty)$ to be cotorsion module over $\Lambda$.

\begin{proposition}\label{prop 2.1}
Let $E$ be an elliptic curve over a number field $K$ and let $p$ be an odd prime number. Let $K_\infty$ be a $\Z_p$-extension of $K$. Assume that the following conditions are satisfied
\begin{enumerate}
\item $E$ has good ordinary reduction at the primes $v|p$,
\item the Selmer group $\op{Sel}_{p^\infty}(E/K)$ is finite.
\end{enumerate}
Then the Selmer group $\op{Sel}_{p^\infty}(E/K_\infty)$ is a cotorsion module over $\Lambda$. 
\end{proposition}
\begin{proof}
The above result is \cite[Corollary 4.9]{greenberg2001introduction}.
\end{proof}

The Selmer group $\op{Sel}_{p^\infty}(E/K_\infty)$ is cofinitely generated over $\Lambda$, and we may define the associated Iwasawa invariants $\mu_p(E/K_\infty)$ and $\lambda_p(E/K_\infty)$ as follows. Setting $M:=\op{Sel}_{p^\infty}(E/K_\infty)$, the module $M^\vee$ is a finitely generated torsion module over $\Lambda$. A map $M_1\rightarrow M_2$ of $\Lambda$-modules is a pseudo-isomorphism if the kernel and cokernel are both finite. We depict this relation by writing $M_1\sim M_2$. A polynomial $f(T)\in \Lambda$ is said to be distinguished if it is monic and its non-leading coefficients are all divisible by $p$. By the structure theorem for finitely generated  torsion modules over $\Lambda$ \cite[Theorem 13.12]{washington1997introduction},
\begin{equation}\label{decomposition}M^\vee\sim \Lambda^r\oplus \left(\bigoplus_{i=1}^s \Lambda/(p^{n_i})\right)\oplus \left( \bigoplus_{j=1}^t \Lambda/(f_j(T)^{m_j})\right),\end{equation}
where $r,s,t,n_i, m_j\in \Z_{\geq 0}$ and $f_j(T)$ are all irreducible distinguished polynomials. The $\mu$-invariant is defined as follows
\[\mu_p(E/K_\infty):=\begin{cases} & 0\text{ if }s=0;\\
& \sum_{i=1}^s n_i \text{ if }s>0.\\
\end{cases}\]
The $\lambda$-invariant is defined as follows 
\[\lambda_p(E/K_\infty):=\begin{cases} & 0\text{ if }t=0;\\
& \sum_{j=1}^t m_j\op{deg}f_j(T) \text{ if }t>0.\\
\end{cases}\]
Suppose that $M$ is cotorsion over $\Lambda$, i.e., $r=0$. It is easy to see that $\mu_p(E/K_\infty)=0$ if and only if $M^\vee$ is finitely generated as a $\Z_p$-module. Furthermore, if these equivalent conditions are satisfied, then,
\begin{equation}\label{lambda is equal to rank}\lambda_p(E/K_\infty)=\op{rank}_{\Z_p} \left(M^\vee\right).\end{equation}

\par From here on in, we assume that $K=\Q(\sqrt{-D})$ is an imaginary quadratic field in which $p$  splits into primes $\p$ and $\p^*$. Here, $D>0$ is a squarefree integer. There are precisely two $\Z_p$-extensions of $K$ that are Galois over $\Q$, namely, the cyclotomic $\Z_p$-extension $K_{\op{cyc}}$ and the anticyclotomic $\Z_p$-extension $\Kant$. The cyclotomic $\Z_p$-extension  is abelian over $\Q$, while the anticyclotomic $\Z_p$-extension is pro-dihedral over $\Q$. Given a finite prime $v$, set $v(\Kant)$ to denote the set of primes $w$ of $\Kant$ that lie above $v$.

\par It conveniences us to work with Selmer groups over $\Kant$ defined by Greenberg's local conditions. It follows (for instance) from \cite[Lemma 6.5]{kunduraykida} that this definition coincides with that of \eqref{selmer group definition alt}. Given a finite prime $v\nmid p$ of $K$, set
\[\cH_v(E/\Kant):=\prod_{w\in v(\Kant)} H^1(K_{\op{anti},w}, E)[p^\infty].\] Here, $K_{\op{anti},w}$ is taken to be the union of local fields $\bigcup_{n\geq 0} K_{n, w}$. Let $v|p$ be a prime, and $T:=T_p(E)$ be the $p$-adic Tate-module of $E$, viewed as a module over $\op{G}_v:=\op{Gal}(\bar{F}_v/F_v)$. Since $E$ is assumed to have good ordinary reduction at $v$, $T$ fits into a short exact sequence of $\op{G}_v$-modules
\[0\rightarrow T^+\rightarrow T\rightarrow T^-\rightarrow 0. \] The modules $T^{\pm}$ are free $\Z_p$-modules of rank $1$, and the action of $\op{G}_v$ on $T^-$ is unramified. We set $D_v$ to denote the $p$-primary $\op{G}_v$-module defined by $D_v:=T^-\otimes_{\Z_p} \left(\Q_p/\Z_p\right)$. For $w\in v(\Kant)$, set 
\[L_w:=\op{ker}\left(H^1(K_{\op{anti}, w}, E[p^\infty])\longrightarrow H^1(\op{I}_w, D_v)\right),\]
where $\op{I}_w$ is the inertia subgroup of the absolute Galois group of $K_{\op{anti}, w}$. The map above is induced from the natural map $E[p^\infty]\rightarrow D_v$ obtained by tensoring $T\rightarrow T^-$ with $\Q_p/\Z_p$. With respect to the above notation, for $v|p$, we define the local Greenberg Selmer condition as follows
\[\cH_v(E/\Kant):=\prod_{w\in v(\Kant)} H^1(K_{\op{anti}, w}, E[p^\infty])/L_w.\]
Note that this product can be infinite. The Selmer group is then described as follows
\begin{equation}\label{greenberg selmer equation}\op{Sel}_{p^\infty} (E/\Kant):=\op{ker} \left\{H^1(K_S/\Kant, E[p^\infty])\rightarrow \bigoplus_{v\in S}\cH_v(E/\Kant)\right\}.\end{equation}

\subsection{Integrally diophantine extensions and Shlapentokh's criterion}

\par In this section, we briefly recall the notion of an integrally diophantine extension. For further details, the reader is referred to \cite[Section 1.2]{shlapentokh2007hilbert}. Let $A$ be a commutative ring with identity, and $A^n$ be the free $A$-module of rank $n$, consisting of tuples $a=(a_1, \dots, a_n)$ with entries in $A$. Let $m$ and $n$ be positive integers and let $a=(a_1, \dots, a_n)\in A^n$ and $b=(b_1, \dots, b_m)\in A^m$. Denote by $(a,b)\in A^{n+m}$ the tuple $(a_1, \dots, a_n, b_1, \dots, b_m)$. Given a finite set of polynomials, $F_1,\dots, F_k$, we set 
\[\cF(a; F_1, \dots, F_k):=\{b\in A^m\mid F_i(a,b)=0\text{ for all }i=1,\dots, k\}.\]
\begin{definition}\label{diophantine subset}A subset $S$ of $A^n$ is a \emph{diophantine subset} of $A^n$ if for some $m\geq 1$, there are polynomials $F_1, \dots, F_k\in A[x_1, \dots, x_n, y_1, \dots, y_m]$ such that $S$ consists of all $a\in A^n$ for which the set $\cF(a; F_1, \dots, F_k)$ is nonempty.
\end{definition}

\begin{definition}
An extension of number fields $L/K$ is said to be \emph{integrally diophantine} if $\cO_K$ is a diophantine subset of $\cO_L$.
\end{definition}
Let $L', L$ and $K$ be number fields such that 
\[L'\supseteq L\supseteq K.\]Suppose that $L'/L$ and $L/K$ are integrally diophantine extensions. Then, it is a well known fact that $L'/K$ is an integrally diophantine extension. Indeed, this is a special case of \cite[Theorem 2.1.15]{shlapentokh2007hilbert}.

\section{Iwasawa theory of congruences}\label{s 3}
\par Building on the results of Greenberg-Vatsal \cite{greenberg2000iwasawa} and Kundu-Ray \cite{kunduraykida}, we describe the effect of congruences between elliptic curves on Selmer groups over anticyclotomic $\Z_p$-extensions. The results apply in what is called the \emph{definite case}, i.e., the setting in which the Selmer groups are cotorsion over the Iwasawa algebra. The results shall be applied to a congruence between an elliptic curve of rank $0$ and an elliptic curve of positive rank. In the setting in which we work, both the Selmer groups over $\Kant$ shall be cofinitely generated as $\Z_p$-modules. We extend the results of \cite[section 6]{kunduraykida} (which are proven for $p\geq 5$) to include the case $p=3$ as well. Let $E_1$ and $E_2$ be elliptic curves defined over $K$. We let $E_i[p]$ denote the $p$-torsion subgroup of $E_i(\bar{K})$, equipped with the natural action of the absolute Galois group $\op{G}_{K}:=\op{Gal}(\bar{K}/K)$. We say that $E_1$ and $E_2$ are $p$-congruent if $E_1[p]$ is isomorphic to $E_2[p]$  as modules over $\op{G}_{K}$. Let $\bar{\rho}_i:\op{G}_{K}\rightarrow \op{GL}_2(\F_p)$ denote the Galois representation associated to $E_i[p]$. Assume that $p$ splits in $K$, and let $\Sigma_p=\{\p, \p^*\}$ be the primes of $K$ that lie above $p$. We assume that $E_1(K)[p]=E_2(K)[p]=0$. In other words, the Galois modules $E_i[p]$ do not contain the trivial module as a submodule. Note that this condition is automatically satisfied when the Galois representation $\bar{\rho}_i$ is irreducible. Since it is assumed that $\bar{\rho}_1$ is isomorphic to $\bar{\rho}_2$, it follows that $E_1(K)[p]=0$ if and only if $E_2(K)[p]=0$. We assume that both elliptic curves $E_1$ and $E_2$ have good ordinary reduction at all primes $v\in \Sigma_p$. Let $\mathfrak{T}$ be the set of primes $v$ of $K$ such that either $E_1$ or $E_2$ (or both) have bad reduction at $v$. Set $K'$ to denote $K(\mu_p)$.

\begin{definition}\label{definition Sigma}
For $i=1,2$, the set $\Sigma(E_i)$ consists of primes $v\in \mathfrak{T}$ such that the following conditions is satisfied.
\begin{enumerate}
\item If the elliptic curve $E_i$ has good reduction at $v$, then, $p\mid \#E_i(\F_v)$.
\item If the elliptic curve $E_i$ has bad reduction at $v$, and $p\geq 5$, then $E_i$ has split multiplicative reduction at some prime $w|v$ of $K'$.
\item If the elliptic curve $E_i$ has bad reduction at $v$, and $p=3$, then, at some prime $w|v$ of $K'$, the elliptic curve $E_i$ has either split multiplicative reduction or has additive reduction of type $\rm{IV}$ or $\rm{IV}^*$. 
\end{enumerate}
We set $\Sigma:=\Sigma(E_1)\cup \Sigma(E_2)$ and $S:=\mathfrak{T}\cup \Sigma_p$.
\end{definition}

 Note that $\Sigma$ is a subset of $S$. For $i=1,2$, we define the imprimitive Selmer group as follows \begin{equation}\label{def of imprimitive selmer}\op{Sel}_{p^\infty}^{\Sigma}(E_i/\Kant):=\op{ker}\left(H^1(\Kant, E_i[p^\infty])\longrightarrow \bigoplus_{v\in S\backslash \Sigma} \cH_v(E_i/\Kant)\right).\end{equation}
In other words, we relax the local conditions at the primes $v\in \Sigma$. There is one crucial difference between the cyclotomic and anticyclotomic extensions. It is easy to see that all primes are finitely decomposed in the cyclotomic $\Z_p$-extension (cf. \cite[Proposition 2]{brink2007prime}). This is not the case for other $\Z_p$-extensions. Let $\ell$ be a prime number and $v|\ell$ be a prime of $K$ that lies above $\ell$. By an application of class field theory, $v$ is infinitely decomposed in $\Kant$ if and only if $\ell\neq p$ and $\ell$ is inert or ramified in $K$ (cf. \cite[p.2132, ll. 8-10]{brink2007prime}). Throughout this section, we make the following assumption.
\begin{ass}\label{main ass section 3}
Assume that all primes $v\in \Sigma$ are finitely decomposed in $\Kant$. Equivalently, we assume that $\ell$ splits into a product of two primes of $K$ for every prime $\ell$ divisible by some $v \in \Sigma$.
\end{ass}
For $v\in S$, we refer to \cite[section 6]{kunduraykida} for the definition of the local condition $\cH_v(E[p]/\Kant)$.
The \emph{residual imprimitive Selmer group} is defined as follows 
\[\op{Sel}_{p^\infty}^{\Sigma} (E[p]/\Kant):=\op{ker}\left\{H^1(K_S/\Kant, E[p^\infty])\rightarrow \bigoplus_{v\in S\backslash \Sigma} \mathcal{H}_v(E[p]/\Kant)\right\}.\]

\begin{lemma}\label{iota map injective}
For $v\in S\backslash \Sigma$, the natural map 
\begin{equation}\label{iota map equation}\mathcal{H}_v(E[p]/\Kant)\rightarrow \mathcal{H}_v(E[p^\infty]/\Kant)[p]\end{equation} is injective.
\end{lemma}
\begin{proof}
When $p\geq 5$, the above result is \cite[Lemma 6.10]{kunduraykida}. Therefore, we consider the case when $p=3$. 
\par First, we assume that $v|p$. In this case, the result follows from part (1) in the proof of Lemma 6.10 from \emph{loc. cit}. Next, suppose that $v\nmid p$ is a prime at which $E$ has good reduction. The injectivity of the map follows from part (2) in the proof of the above mentioned Lemma. The only case to consider therefore is when $v\nmid p$ is a prime at which $E$ has bad reduction. Recall that $K'=K(\mu_p)$, and set $K'_{\op{anti}}=K'\cdot K_{\op{anti}}$. It follows from the proof of (3) of Lemma 6.10 of \emph{loc. cit.} that the map \eqref{iota map equation} is injective if $E(K'_{\op{anti}, w'})[p^\infty]=0$ for all primes $w'|v$ of $K'_{\op{anti}}$. Let $w$ be a prime of $K'$ such that $w|v$. Let $w'|w$ be a prime of $K'_{\op{anti}}$. Since $K'_{\op{anti},w'}/K'_w$ is a pro-$p$ extension, it follows from \cite[Lemma 6.9]{kunduraykida} that \[E(K'_w)[p^\infty]=0\Rightarrow E(K'_{\op{anti}, w'})[p^\infty]=0.\]
Therefore, it suffices to show that $E(K'_w)[p^\infty]=0$ for all $w|v$ of $K'$. We recall part (3) of the Definition \ref{definition Sigma}. Note that in this case $p=3$, $v\notin \Sigma$ and $E$ is assumed to have bad reduction at $v$. Therefore, $E$ has non-split multiplicative reduction or additive reduction of Kodaira type $I\notin \{\rm{IV}, \rm{IV}^*\}$ at all primes $w|v$ of $K'$. Let $E_0(K'_w)$ be the subgroup of $E$ consisting of points of non-singular reduction. It follows from the proof of \cite[Proposition 5.1 (iii)]{hachimori1999analogue} that $E_0(K'_w)[p^\infty]=0$. It suffices to show that $p\nmid [E(K'_w):E_0(K'_w)]$ for all primes $w|v$ of $K'$. Since $p=3$, and $E$ has non-split multiplicative reduction or additive reduction of Kodaira type $I\notin \{\rm{IV}, \rm{IV}^*\}$, the result follows from \cite[4th row of Table 4.1, p. 365]{silverman1994advanced}.  
\end{proof}

\begin{proposition}\label{prop 3.3}
Let $E$ be either $E_1$ or $E_2$ and recall that $\Sigma=\Sigma(E_1)\cup \Sigma(E_2)$. Assume that $E(K)[p]=0$. Then, the natural map 
\[\op{Sel}_{p^\infty}^{\Sigma} (E[p]/\Kant)\xrightarrow{\sim} \op{Sel}_{p^\infty}^{\Sigma} (E/\Kant)[p]\] is an isomorphism. 
\end{proposition}

\begin{proof}
The result follows verbatim from the proof of \cite[Proposition 6.11]{kunduraykida} and Lemma \ref{iota map injective}. 
\end{proof}
\par We introduce some further notation.  Let $E$ be either $E_1$ or $E_2$. Recall that by the Assumption \ref{main ass section 3}, each prime $v\in \Sigma$ is finitely decomposed in $\Kant$. For $v\in \Sigma$, let $\ell$ be the prime number for which $v|\ell$. Note that $\ell\neq p$, since both elliptic curves $E_1$ and $E_2$ have good reduction at all primes $v\in \Sigma_p$. Denote by $v^*$ the conjugate prime, so that $\ell\cO_K=v v^*$. We identify $K_v$ with $\Q_\ell$, and set $\Q_\ell^{\op{nr}}$ to denote the unramified $\Z_p$-extension of $\Q_\ell$. Recall that $s_v$ is the number of primes of $\Kant$ that lie above $v$. We note that \cite[Theorem 2, p. 2133]{brink2007prime} gives an explicit algorithm to compute $s_v$. Let $\sigma_{E}^{(v)}$ denote the $\Z_p$-corank of $\mathcal{H}_v(E/\Kant)$, for $v\in \Sigma$.  Recall that $d_v(E)$ is the multiplicity of $\ell^{-1}$ as a root of the local $L$-factor $P_v(X)$.


\begin{proposition}\label{computing sigma_v}
For $v\in \Sigma$, the $\mu$-invariant of $\cH_v(E/\Kant)$ is $0$ and $\sigma_E^{(v)}=s_vd_v(E)$.
\end{proposition}
\begin{proof}
The above result follows from the proof of \cite[Proposition 2.4]{greenberg2000iwasawa}.
\end{proof}
\begin{theorem}\label{thm 3.5}
Let $p$ be an odd prime and $K$ be an imaginary quadratic field in which $p$ splits. Let $E_1$ and $E_2$ be two elliptic curves defined over $K$. Assume that the following conditions hold
\begin{enumerate}
\item Both elliptic curves have good ordinary reduction at the primes $v\in\Sigma_p$.
\item The elliptic curves $E_1$ and $E_2$ are $p$-congruent. In other words, $E_1[p]$ and $E_2[p]$ are isomorphic as modules over $\op{G}_K$.
\item We require that $E_1(K)[p]=E_2(K)[p]=0$. (This assumption is clearly satisfied when the residual representations $\bar{\rho}_i$ on $E_i[p]$ are irreducible as modules over $\op{G}_K$.)
\item Assumption \ref{main ass section 3} is satified. In other words, all primes $v\in \Sigma$ (cf. Definition \ref{definition Sigma}) are finitely decomposed in $\Kant$.
\end{enumerate}Then, the Selmer group $\op{Sel}_{p^\infty}(E_1/\Kant)$ is cotorsion over $\Lambda$ with $\mu=0$ if and only if $\op{Sel}_{p^\infty}(E_2/\Kant)$ is cotorsion over $\Lambda$ with $\mu=0$. Furthermore, suppose that indeed both Selmer groups $\op{Sel}_{p^\infty}(E_i/\Kant)$ are cotorsion over $\Lambda$ with $\mu=0$. Then, we have that 
\[\lambda_p(E_2/\Kant)-\lambda_p(E_1/\Kant)=\sum_{v\in \Sigma} \left(\sigma_{E_1}^{(v)}-\sigma_{E_2}^{(v)}\right)=\sum_{v\in \Sigma} s_v\left(d_v(E_1)-d_v(E_2)\right).\]
\end{theorem}
\begin{proof}
When $p\geq 5$, the above result follows from \cite[Theorem 6.14, Proposition 6.19]{kunduraykida}. When $p=3$, the same proof applies verbatim from Proposition \ref{prop 3.3}, the key ingredient being Lemma \ref{iota map injective}. We provide a sketch of the proof. 
\par Let $M$ be a cofinitely generated $\Lambda$-module with $\mu$-invariant $\mu(M)$ and $\lambda$-invariant $\lambda(M)$. Then, we recall that there is a decomposition map \eqref{decomposition} 
\[\Phi:M^\vee\rightarrow  \Lambda^r\oplus \left(\bigoplus_{i=1}^s \Lambda/(p^{n_i})\right)\oplus \left( \bigoplus_{j=1}^t \Lambda/(f_j(T)^{m_j})\right)\]
which is a map of $\Lambda$-modules whose kernel and cokernel are finite. The $\mu$-invariant $\mu(M)$ is defined to be the sum $\sum_i n_i$ and the $\lambda$-invariant is $\lambda(M):=\sum_j m_j \op{deg} f_j$.
It follows that
\[M\text{ is cotorsion with }\mu(M)=0\Leftrightarrow M \text{ is cofinitely generated as a }\Z_p\text{-module.}\]Moreover, if these conditions are satisfied, then, \[\op{corank}_{\Z_p} M=\lambda(M).\]We note that $M$ is cofinitely generated as a $\Z_p$-module if and only if $M[p]$ is finite. Assume that $M$ has no proper $\Lambda$-submodules of finite index, then, $M^\vee$ has no non-zero finite $\Lambda$-submodules. Then, since the kernel of $\Phi$ is finite, it must be $0$. Since $M^\vee$ is torsion, it follows that $r=0$, and hence the target of $\Phi$ is a free $\Z_p$-module of rank $\lambda(M)$. Since $\op{cok}\Phi$ is finite and $\Phi$ is injective, it follows that $M^\vee$ is a free $\Z_p$-module of rank equal to $\lambda(M)$. In particular, we find that 
\[\op{dim}_{\F_p}M[p]=\lambda(M).\]
\par We specialize the above discussion to the $\Sigma$-imprimitive Selmer groups \[M_i:=\op{Sel}_{p^\infty}^{\Sigma} (E_i[p]/\Kant).\]
It follows from Proposition \ref{prop 3.3} that 
\[M_i[p]\simeq \op{Sel}^\Sigma(E_i[p]/\Kant).\] It thus follows that $M_1[p]\simeq M_2[p]$, for further details, cf. \cite[Lemma 6.12]{kunduraykida}. Take $\mu_p^\Sigma(E_i/\Kant)$ (resp. $\lambda_p^\Sigma(E_i/\Kant)$) to denote $\mu(M_i)$ (resp. $\lambda(M_i)$). It is shown that $M_i^\vee$ has no non-zero finite $\Lambda$-submodules, cf. \cite[Proposition 6.18]{kunduraykida}. It follows from the discussion above that 
\[\begin{split}& \op{Sel}_{p^\infty}^{\Sigma} (E_1/\Kant)\text{ is cotorsion with }\mu_p^\Sigma(E_1/\Kant)=0 \\
\Leftrightarrow & \op{Sel}_{p^\infty}^{\Sigma} (E_2/\Kant)\text{ is cotorsion with }\mu_p^\Sigma(E_2/\Kant)=0. \end{split}\]
Assume that the above equivalent conditions are satisfied. The equality $\dim_{\F_p}M_1[p]=\dim_{\F_p}M_2[p]$ implies that 
\[\lambda_p^\Sigma(E_1/\Kant)=\lambda_p^\Sigma(E_2/\Kant).\]
Next, we relate the imprimitive Iwasawa invariants $\mu_p^\Sigma(E_i/\Kant)$ and $\lambda_p^\Sigma(E_i/\Kant)$ to the primitive Iwasawa invariants $\mu_p(E_i/\Kant)$ and $\lambda_p(E_i/\Kant)$. There is a natural short exact sequence
\[0\rightarrow \op{Sel}_{p^\infty}(E_i/\Kant)\rightarrow \op{Sel}_{p^\infty}^{\Sigma}(E_i/\Kant)\rightarrow \bigoplus_{v\in \Sigma} \mathcal{H}_v(E_i/\Kant)\rightarrow 0,\]cf. \cite[Eqn. (6.1) and Lemma 3.4]{kunduraykida}. Proposition \ref{computing sigma_v} asserts that the $\mu$-invariant of $\cH_v(E_i/\Kant)$ is $0$ and $\sigma_{E_i}^{(v)}=s_vd_v(E_i)$. It follows that 
\[\mu_p^\Sigma(E_i/\Kant)=\mu_p(E_i/\Kant)\]
and 
\[\begin{split}
\lambda_p^\Sigma(E_i/\Kant)=& \corank_{\Z_p}\op{Sel}_{p^\infty}^{\Sigma}(E_i/\Kant) \\
= & \corank_{\Z_p}\op{Sel}_{p^\infty}(E_i/\Kant)+ \sum_{v\in \Sigma} \corank_{\Z_p}\left(\mathcal{H}_v(E_i/\Kant)\right) \\
= &\lambda_p(E_i/\Kant)+\sum_{v\in \Sigma} \sigma_{E_i}^{(v)}.
\end{split}\]
From the above relations, we obtain the result.
\end{proof}

\section{A general criterion}\label{s 4}
\par Let $p$ be an odd prime, $K=\Q(\sqrt{-D})$ be an imaginary quadratic field in which $p$ splits, and $\Kant$ denote the anticyclotomic $\Z_p$-extension of $K$. Given $n\in \Z_{\geq 1}$, we recall that $K_n$ is $n$-th layer in $\Kant$. Let $\p$ and $\p^*$ be the primes of $K$ that lie above $p$.

\begin{theorem}\label{first theorem section 3}
With respect to notation above, suppose that there exists an elliptic curve $E_{/K}$ such that all of the following conditions are satisfied
\begin{enumerate}
\item\label{c1 first theorem section 3} the elliptic curve $E$ has good ordinary reduction at the primes $\p$ and $\p^*$, 
\item\label{c2 first theorem section 3} the Selmer group $\op{Sel}_{p^\infty}(E/\Kant)$ is cotorsion over $\Lambda$  with $\mu_p(E/\Kant)=0$, 
\item\label{c3 first theorem section 3} the rank of the Mordell-Weil group $E(K)$ is positive, 
\item $\lambda_p(E/\Kant)=\op{rank} E(K)$.
\end{enumerate}
Then, for all $n\in \Z_{\geq 0}$, the Hilbert's tenth problem has a negative answer for $\cO_{K_n}$.
\end{theorem}
\begin{proof}
It follows from the structure theory of $\Lambda$-modules that condition \eqref{c2 first theorem section 3} above implies that $\op{Sel}_{p^\infty}(E/\Kant)^{\vee}$ is a finitely generated $\Z_p$-module and
\[\lambda_p(E/\Kant)=\op{rank}_{\Z_p}\left(\op{Sel}_{p^\infty}(E/\Kant)^{\vee}\right).\]
The kernel of the natural restriction map
\[\op{Sel}_{p^\infty}(E/K_n)\rightarrow \op{Sel}_{p^\infty}(E/\Kant)\] is finite (cf. \cite[Theorem 4.1]{greenberg2001introduction}), and therefore, 
\[\op{rank}_{\Z_p}\left( \op{Sel}_{p^\infty}(E/K_n)^\vee\right)\leq \lambda_p(E/\Kant).\]
Note that $\op{Sel}_{p^\infty}(E/K_n)$ fits into a short exact sequence 
\[0\rightarrow E(K_n)\otimes \Q_p/\Z_p\rightarrow \op{Sel}_{p^\infty}(E/K_n)\rightarrow \Sh(E/K_n)[p^\infty]\rightarrow 0,\] where $\Sh(E/K_n)[p^\infty]$ is the $p$-primary part of the Tate-Shafarevich group of $E/K_n$. Therefore, we deduce that for all $n\geq 0$,
\[\op{rank}E(K_n)\leq \op{rank}_{\Z_p}\left( \op{Sel}_{p^\infty}(E/K_n)^\vee\right)\leq \lambda_p(E/\Kant).\]
On the other hand, it is assumed that $\lambda_p(E/\Kant)=\op{rank}E(K)$. Therefore, we find that $\op{rank}E(K_n)=\op{rank}E(K)$. Since, $\op{rank}E(K)>0$, it follows from Theorem \ref{shlap 1} that $K_n/K$ is integrally diophantine. Since $K/\Q$ is integrally diophantine, it follows that for all $n$, $K_n/\Q$ is integrally diophantine. Therefore, for all $n$, Hilbert's tenth problem has a negative answer for $\cO_{K_n}$.
\end{proof}

In order to construct an elliptic curve $E$ that satisfies the above conditions, we begin with two elliptic curves $E_1$ and $E_2$ that are both defined over the rational numbers. We assume that $E_1$ and $E_2$ are $p$-congruent, i.e., the Galois modules $E_1[p]$ and $E_2[p]$ are isomorphic.
\begin{proposition}\label{prop 3.2}
Let $K$ be an imaginary quadratic field and $p$ be an odd prime number which splits into primes $\p$ and $\p^*$ in $\cO_K$. Assume that there exist elliptic curves $E_1$ and $E_2$ defined over $K$ such that 
\begin{enumerate}
\item both $E_1$ and $E_2$ have good ordinary reduction at $\p$ and $\p^*$,
\item $E_1[p]$ and $E_2[p]$ are isomorphic as modules over $\op{Gal}(\bar{K}/K)$,
\item $\op{rank} E_1(K)=0$ and $\op{rank}E_2(K)>0$, 
\item $E_1(K)[p]=E_2(K)[p]=0$, 
\item each prime $v\in \Sigma$ is finitely decomposed in the anticyclotomic extension $\Kant/K$. Equivalently, let $v\in \Sigma$ and $\ell$ be the prime number for which $v|\ell$. Then, $\ell$ splits in $K$ into a product of two primes $v$ and $v^*$.
\item\label{c6 of prop 3.2} the Selmer group $\op{Sel}_{p^\infty}(E_1/\Kant)$ is cotorsion over $\Lambda$ with $\mu_p(E_1/\Kant)=0$ and $\lambda_p(E_1/\Kant)=0$.
 
\end{enumerate}
Suppose that the following relationship holds
\begin{equation}\label{main equation for rank}\op{rank} E_2(K)=\sum_{v\in \Sigma} \left(\sigma_{E_1}^{(v)}-\sigma_{E_2}^{(v)}\right),\end{equation}
where we recall that $\sigma_{E}^{(v)}$ is the $\Z_p$-corank of $\mathcal{H}_v(E/\Kant)$,

Then, for all $n\in \Z_{\geq 0}$, Hilbert's tenth problem has a negative answer for $\cO_{K_n}$.
\end{proposition}
\begin{proof}
Let $E$ be the elliptic curve $E_2$. We show that the conditions of Theorem \ref{first theorem section 3} is satisfied. The conditions \eqref{c1 first theorem section 3} and \eqref{c3 first theorem section 3} of the Theorem are satisfied by assumption. The condition \eqref{c2 first theorem section 3} of the Theorem requires that the Selmer group $\op{Sel}_{p^\infty}(E/\Kant)$ is cotorsion over $\Lambda$  with $\mu_p(E/\Kant)=0$. We note that it is assumed that $\op{Sel}_{p^\infty}(E_1/\Kant)$ is cotorsion over $\Lambda$ with $\mu_p(E_1/\Kant)=0$. Therefore, the condition \eqref{c2 first theorem section 3} follows from Theorem \ref{thm 3.5}. Furthermore, it follows from Theorem \ref{thm 3.5} that 
\[\lambda_p(E/\Kant)-\lambda_p(E_1/\Kant)=\sum_{v\in \Sigma} \left(\sigma_{E_1}^{(v)}-\sigma_{E}^{(v)}\right).\]
Since it is assumed that 
\[\lambda_p(E_1/\Kant)=0\]
and that 
\[\op{rank} E(K)=\sum_{v\in \Sigma} \left(\sigma_{E_1}^{(v)}-\sigma_{E}^{(v)}\right),\]
we deduce that 
\[\op{rank}E(K)=\lambda_p(E/\Kant).\] The assumptions of Theorem \ref{first theorem section 3} are satisfied and the result follows.
\end{proof}

The condition \eqref{c6 of prop 3.2} of Proposition \ref{prop 3.2} is the most difficult to verify and warrants further explanation. We derive explicit conditions on the elliptic curve $E_1$ for this condition to be satisfied. These conditions are obtained by studying the Euler characteristic of the Selmer group $\op{Sel}_{p^\infty}(E/\Kant)$, where $E=E_1$. 

 \par Let $M$ be a cofinitely generated cotorsion $\Lambda$-module. Since the $p$-cohomological dimension of $\Lambda$ is $0$, it follows that $H^i(\Gamma, M)=0$ for all values $i\geq 2$. Note that $H^1(\Gamma, M)$ is identified with $M_{\Gamma}$. The module $H^0(\Gamma, M)$ is finite if and only if $H^1(\Gamma, M)$ is finite (cf. \cite[Lemma 2.1]{kunduraystats}). The Euler characteristic of $M$ is then defined as follows
 \[\chi(\Gamma, M):=\frac{\# H^0(\Gamma, M)}{\# H^1(\Gamma, M)}. \]
 Let $f_M(T)=\sum_i c_iT^i$ be a generator of the characteristic ideal of $M^\vee$. Given $a,b\in \Z_p$, we write $a\sim b$ if $a=ub$ for some unit $u\in \Z_p^\times$. 

 \begin{lemma}\label{tiny lemma}
 Let $M$ be a cofinitely generated cotorsion $\Lambda$-module such that the cohomology groups $H^0(\Gamma, M)$ and $H^1(\Gamma, M)$ are finite. Then, $\chi(\Gamma, M)\sim c_0$.
 \end{lemma}
 \begin{proof}
 This is a well known result, we refer for instance to \cite[Lemma 3.4]{hatley2021statistics}.
 \end{proof}

\begin{proposition}\label{prop 4.4}
Let $M$ be a cofinitely generated and cotorsion $\Lambda$-module for which the cohomology groups $H^i(\Gamma, M)$ are finite.
Then, the following are equivalent
\begin{enumerate}
 \item\label{261} $\chi(\Gamma, M)=1$,
 \item\label{262} $f_{M}(T)=1$,
 \item\label{263} $\mu(M)=0$ and $\lambda(M)=0$,
 \item\label{264} $M$ has finite cardinality.
\end{enumerate} 
\end{proposition}
\begin{proof}
The above result is a direct consequence of Lemma \ref{tiny lemma}. We refer to \cite[Proposition 3.6]{hatley2021statistics} for the proof.
\end{proof}

\begin{theorem}\label{thm 4.5}
Let $E$ be an elliptic curve defined over an imaginary quadratic field $K$. Set $\Gamma$ to denote the Galois group $\op{Gal}(\Kant/K)$. Assume that the following conditions are satisfied
\begin{enumerate}
\item $E(K)[p]=0$,
\item $E$ has good ordinary reduction at each prime that lies above $p$.
\item $\op{rank}E(K)=0$,
 \item $\Sh(E/K)[p^{\infty}]$ is finite.
\end{enumerate} 
Then, the following assertions hold.
\begin{enumerate}
\item the Selmer group $M:=\op{Sel}_{p^\infty}(E/\Kant)$ is a cotorsion module over $\Lambda$, 
\item the cohomology groups $H^0(\Gamma, M)$ and $H^1(\Gamma, M)$ are finite, 
\item the Euler characteristic is given by the following formula
\[\chi(\Gamma, E[p^{\infty}])\sim \# \Sh(E/K)[p^{\infty}]\times \left(\prod_{\p|p}\# E(\F_\p)\right)^2\times \prod_{v\in \Omega_K} c_v^{(p)}(E).\]

\end{enumerate}

\end{theorem}
\begin{proof}
That $M$ is a cotorsion over $\Lambda$ follows from Proposition \ref{prop 2.1}. That the cohomology group $H^i(\Gamma, M)$ are finite, follows from \cite[Lemma 3.3]{hatley2021statistics}. The formula for the Euler characteristic is due to J. Order \cite{van15}. There is a correction to this formula, for which we refer to \cite[Theorem 4.4]{hatley2021statistics}. 
\end{proof}

\begin{corollary}\label{cor 4.6}
Let $E$ be an elliptic curve satisfying the above conditions of Theorem \ref{thm 4.5}. Furthermore, assume that the following additional conditions are satisfied
\begin{enumerate}
\item $\Sh(E/K)[p^\infty]=0$, 
\item $E(\F_{\p})[p]=0$, 
\item $c_v^{(p)}(E)=1$ for all $v\in \Omega_K$. 
\end{enumerate}
Then, the Selmer group $\op{Sel}_{p^\infty}(E/\Kant)$ is a cotorsion module over $\Lambda$, with $\mu_p(E/\Kant)=\lambda_p(E/\Kant)=0$. 
\end{corollary}
\begin{proof}
It follows from Theorem \ref{thm 4.5} that $\chi\left(\Gamma, \op{Sel}_{p^\infty}(E/\Kant)\right)=1$. Thus the result follows from Proposition \ref{prop 4.4}.
\end{proof}

\begin{theorem}\label{section 3 main thm}
Suppose that $E_1$ and $E_2$ satisfy the conditions (1)-(5) of Proposition \ref{prop 3.2}, and that 
\begin{enumerate}
\item $\Sh(E_1/K)[p^\infty]=0$, 
\item $E_1(\F_{\p})[p]=0$, 
\item $c_v^{(p)}(E_1)=1$ for all $v\in \Omega_K$,
\item \begin{equation}\label{main equation for rank}\op{rank} E_2(K)=\sum_{v\in \Sigma} \left(\sigma_{E_1}^{(v)}-\sigma_{E_2}^{(v)}\right),\end{equation}
\end{enumerate}
Then, for all $n\in \Z_{\geq 0}$, the Hilbert's tenth problem has a negative answer for the ring of integers of $K_n$.
\end{theorem}
\begin{proof}
The result is deduced from Proposition \ref{prop 3.2}. It follows from Corollary \ref{cor 4.6} that the condition \eqref{c6 of prop 3.2} of this Proposition is satisfied. We assume that the conditions (1)-(5) of Proposition \ref{prop 3.2} are satisfied. Hence, the result follows.
\end{proof}
\begin{proof}[Proof of Theorem \ref{main theorem}]
Let $E_1$ and $E_2$ be elliptic curves which satisfy the conditions of the Theorem \ref{main theorem}. Note that by Proposition \ref{computing sigma_v}, $\sigma_{E_i}^{(v)}=s_vd_v(E_i)$. The conditions of Theorem \ref{section 3 main thm} are satisfied, and the result follows.
\end{proof}
\section{An explicit example}\label{s 5}
\par In this section, we construct an explicit congruence between two elliptic curves $E_1$ and $E_2$ defined over $\Q$ for which the conditions of Theorem \ref{main theorem} are satisfied. It suffices to check that the conditions of Theorem \ref{section 3 main thm} are satisfied. We shall take $K=\Q(\sqrt{-5})$ and $p=3$ and show that the Hilbert's tenth problem has a negative answer in $\cO_{K_n}$ for all $n\geq 0$, in the anticyclotomic $\Z_3$-extension of $K$. This will thus prove Theorem \ref{theorem main aux}. Let $E_1$ be \href{https://www.lmfdb.org/EllipticCurve/Q/56/b/2}{56b1} and $E_2$ be \href{https://www.lmfdb.org/EllipticCurve/Q/392/c/1}{392c1} with respect to Cremona's labelling for elliptic curves. Both elliptic curves are defined over $\Q$, with good ordinary reduction at $3$. The elliptic curves $E_1$ and $E_2$ are associated to normalized newforms $f_1(z)=\sum_{n=1}^\infty a_n(f_1) q^n$ and $f_2(z)=\sum_{n=1}^\infty a_n(f_2) q^n$ respectively of weight $2$ on $\Gamma_0(392)$. Here, $q$ is taken to be $e^{2\pi i z}$. At $p=3$, the residual representations for both $E_1$ and $E_2$ are surjective. \par This can be checked using the following code on the {\tt SageMathCloud}:
\vspace{0.5cm} \begin{verbatim}
E1=EllipticCurve(`56b1')
E2=EllipticCurve(`392c1')
rho1=E1.galois_representation()
rho2=E2.galois_representation()
rho1.image_type(3)
rho2.image_type(3)
\end{verbatim}\vspace{0.5cm} 
Next, we show that the elliptic curves $E_1$ and $E_2$ are $3$-congruent, i.e., $E_1[3]\simeq E_2[3]$ as modules over $\op{Gal}(\bar{\Q}/\Q)$. Since the residual Galois representations at $p=3$ are irreducible, by the Brauer-Nesbitt theorem, it suffices to show that 
\[a_n(f_1)\equiv a_n(f_2)\mod{3}\] for all $n$ coprime to $392$ ($=2^3\times 7^2$). We wish to apply a result of Sturm on congruences between modular forms, and it conveniences us to modify $f_1$ slightly. We set $g_1(z):=f_1(z)-f_1(7z)$, and $g_2(q):=f_2(z)$. Let $a_n(g_i)$ denote the $n$-th Fourier coefficient of $g_i$. Note that $a_n(g_1)=a_n(f_1)$ for all $n$ coprime to $7$. Therefore, in order to show that \[a_n(f_1)\equiv a_n(f_2)\mod{3}\] for all $n$ coprime to $392$, it suffices to show that 
\[a_n(g_1)\equiv a_n(g_2)\mod{3}\] for all $n$. It follows from the result of Sturm \cite{sturm2006congruence} that if this congruence holds for all $n\leq \frac{[\op{SL}_2(\Z):\Gamma_1(392)]}{12}+1$, then, it holds for all $n$. The index of $\Gamma_1(392)$ is given by
\[[\op{SL}_2(\Z):\Gamma_1(392)]=392^2\left(1-\frac{1}{2^2}\right)\left(1-\frac{1}{7^2}\right)=112896.\]
It suffices to check the congruence $a_n(f_1)\equiv a_n(f_2)\mod{p}$ for all $n\leq 112897$. Continuing with the code on {\tt SageMath}, the following returns {\tt True} for the value of the boolean variable {\tt v}.

\vspace{0.5cm} \begin{verbatim}
v=True
for n in range(112898):
    a=E1.an(n)
    b=E2.an(n)
    if n%7==0:
        m=n/7
        a2=E1.an(m)
        a=E1.an(n)-a2
    c=a-b
    d=c%3
    if d!=0:
        v=False
\end{verbatim}\vspace{0.5cm}

Note that $3$ splits in $K$. We check the conditions (1)-(5) of Proposition \ref{prop 3.2}.

\begin{enumerate}
\item Both $E_1$ and $E_2$ have good ordinary reduction at $\p$ and $\p^*$, 
\item $E_1$ and $E_2$ are $3$-congruent, as has been checked by the code above.
\item The following {\tt SageMath} code verifies that $\op{rank} E_1(K)=0$, $\op{rank} E_2(K)=2$:
\vspace{0.5cm} \begin{verbatim}
K = QuadraticField(-5, `a')
A1=E1.base_extend(K)
A1.rank()
A2=E2.base_extend(K)
A2.rank()
\end{verbatim}\vspace{0.5cm} 
\item Since the residual representations 
\[\bar{\rho}_i:\op{Gal}(\bar{\Q}/\Q)\rightarrow \op{Aut}(E_i[3])\xrightarrow{\sim} \op{GL}_2(\F_3)\]are surjective, they remain irreducble when restricted to the index-$2$ subgroup $\op{Gal}(\bar{\Q}/K)$. In particular, $E_i[p]$ does not contain a trivial submodule for the action of $\op{Gal}(\bar{\Q}/K)$. It follows that for $i=1,2$, $E_i(K)[3]=0$.
\item Both $E_1$ and $E_2$ have additive reduction at the prime above $2$ in $K':=K(\mu_3)$, of Kodaira type $\op{I}_1^*$. Therefore, in accordance with Definition \ref{definition Sigma}, the prime above $2$ is not in $\Sigma$. We find that $\Sigma:=\Sigma(E_1)\cup \Sigma(E_2)=\{v_7, v_7^*\}$, where $v_7$ and $v_7^*$ are the primes of $K$ that lie above $7$.  The prime $7$ splits in $K$, and hence the primes $v_7$ and $v_7^*$ are finitely decomposed in $\Kant$. The code is given below, and is a continuation of our {\tt SageMath} code. The field $L$ in the code below is $K'$.

\vspace{0.5cm} \begin{verbatim}
K.<a> = NumberField(x^2+5)
R.<t> = K[]
L.<b> = K.extension(t^2+t+1)
A1=E1.base_extend(L)
A1.local_data()
A2=E2.base_extend(L)
A2.local_data()
\end{verbatim}\vspace{0.5cm} 
\end{enumerate}
The conditions of Theorem \ref{section 3 main thm} are satisfied, in greater detail,
\begin{enumerate}
\item we show that $\Sh(E_1/K)[3^\infty]=0$ by computing the $3$-Selmer group of $E_1$ over $K$ and showing that it equals $0$. We denote this Selmer group by $\op{Sel}_3(E_1/K)$. Let $E_1'$ denote the twist of $E_1$ by $-5$. The code below shows that $\op{Sel}_3(E_1/\Q)$ and $\op{Sel}_3(E_1'/\Q)$ are both $0$.
\begin{verbatim}
E := EllipticCurve("56b1");
ThreeSelmerGroup(E);
Et := EllipticCurve([0,5,0,0,500]);
ThreeSelmerGroup(Et);
\end{verbatim}
The output tells us that these Selmer groups are abelian groups of order $1$. The commands above use $3$-descent and do not assume the strong form of the Birch and Swinnerton-Dyer conjecture. Since the $3$-Selmer groups of $E_1$ and $E_1'$ over $\Q$ are both $0$, it follows that $\op{Sel}_3(E_1/K)=0$ (cf. \cite[Lemma 3.1]{OnoPapa}).
\item The prime $p=3$ splits in $K$, hence, $\F_{\p}\simeq \F_3$. We find that $a_3(E_1)=2$, hence, $\#E_1(\F_3)=3+1-2=2$. Thus, $E_1(\F_3)[3]=0$ and the condition is satisfied. 
\item The Tamagawa product for $E_1$ is $2$ and hence, no Tamagawa number is divisible by $3$.
\item
The values $d_v(E_2)=0$ for all $v\in \Sigma=\{v_7, v_7^*\}$. This is because $E_2$ has additive reduction at the primes $v_7$ and $v_7^*$, and the polynomial $P_v^{E_2}(X)=1$ for $v\in \{v_7, v_7^*\}$. Thus, the quantity $\sigma_{E_2}^{(v)}=0$ for $v\in \Sigma$. The same reasoning shows that $\sigma_{E_2}^{(v_2)}=0$, where $v_2$ is the prime of $K$ above $2$. 
\par We now show that the algorithm \cite[Theorem 2, p. 2133]{brink2007prime} implies that the primes $v_7$ and $v_7^*$ are totally inert in $\Kant$. This will show that $s_{v_7}=s_{v_7^*}=1$. 
 In accordance with notation from \emph{loc. cit.}, set $\Delta:=5$, note that $K=\Q(\sqrt{-\Delta})$. The number $h$ is the class number of $K$, which we express as $h=3^\mu u$, where $3\nmid u$. The quantity $\mu$ should not be confused with the $\mu$-invariant. We thus find that $h=u=2$ and $\mu=0$. Since $\Delta\not\equiv 3\mod{4}$, the algorithm \cite[Theorem 2, p. 2133]{brink2007prime} asks that write $7^2=a^2+5b^2$, where $a$ and $b$ are non-zero integers that are relatively prime. We find that $a=2$ and $b=3$. Set $\omega:=\sqrt{-\Delta}=\sqrt{-5}$ and find that
 \[(a+b\omega)^{3-1}=(2+3\sqrt{-5})^2=-41+12\sqrt{-5}=a^*+\omega b^*,\]
 where $a^*=-41$ and $b^*=12$.
 The number $\nu$ is defined such that $K_H\cap \Kant=K_\nu$, where $K_H$ is the Hilbert class field of $K$. Since the class number is prime to $p$, we find that $\nu=0$. Part (a) of the aforementioned algorithm tells us that a prime $v\in \{v_7, v_7^*\}$ splits in $K_n$ if and only if 
 \[b^*\equiv 0\mod{3^{(n+1+\mu-\nu)}}.\] Since both $\mu$ and $\nu$ are $0$, and $b^*=12$ is not divisible by $3^2$, we find that $v$ must be inert in every field $K_n$, and thus, inert in $\Kant$.

\par Note that $E_1$ has split multiplicative reduction at both $v_7$ and $v_7^*$. For $v\in \{v_7, v_7^*\}$, we find that $\widetilde{P_v^{E_1}}(X)=1-X\in \Z/3\Z[X]$, and thus, $7^{-1}=1$ is a simple root in $\Z/3\Z$. As a result, $d_v(E_1)=1$, and we find that $\sigma_{E_1}^{(v_7)}=\sigma_{E_1}^{(v_7^*)}=1$. We find that 
\[2=\op{rank} E_2(K)=\sum_{v\in \Sigma} \left(\sigma_{E_1}^{(v)}-\sigma_{E_2}^{(v)}\right).\]
\end{enumerate}

\begin{proof}[Proof of Theorem \ref{theorem main aux}]
We have checked that all conditions of Theorem \ref{section 3 main thm} are satisfied for $(E_1, E_2, K,p)$, and thus, the result follows.
\end{proof}
\bibliographystyle{alpha}
\bibliography{references}

\begin{thebibliography}{MRL18}

\bibitem[Bri07]{brink2007prime}
David Brink.
\newblock Prime decomposition in the anti-cyclotomic extension.
\newblock {\em Mathematics of Computation}, 76(260):2127--2138, 2007.

\bibitem[Coa00]{coates2000galois}
John~Henry Coates.
\newblock Galois cohomology of elliptic curves.
\newblock {\em Lecture Notes at the Tata Institute of Fundamental Research No.
  88}, 2000.

\bibitem[CPZ05]{cornelissen2005division}
Gunther Cornelissen, Thanases Pheidas, and Karim Zahidi.
\newblock Division-ample sets and the diophantine problem for rings of
  integers.
\newblock {\em Journal de th{\'e}orie des nombres de Bordeaux}, 17(3):727--735,
  2005.

\bibitem[CV07]{cornut2007nontriviality}
Christophe Cornut and Vinayak Vatsal.
\newblock Nontriviality of rankin-selberg l-functions and cm points.
\newblock {\em London Mathematical Society Lecture Note Series}, 320:121, 2007.

\bibitem[Den80]{denef1980diophantine}
Jan Denef.
\newblock Diophantine sets over algebraic integer rings. {II}.
\newblock {\em Transactions of the American Mathematical Society},
  257(1):227--236, 1980.

\bibitem[DL78]{denef1978diophantine}
Jan Denef and Leonard Lipshitz.
\newblock Diophantine sets over some rings of algebraic integers.
\newblock {\em Journal of the London Mathematical Society}, 2(3):385--391,
  1978.

\bibitem[DPR61]{davis1961decision}
Martin Davis, Hilary Putnam, and Julia Robinson.
\newblock The decision problem for exponential diophantine equations.
\newblock {\em Annals of Mathematics}, pages 425--436, 1961.

\bibitem[GFP20]{garcia2020towards}
Natalia Garcia-Fritz and Hector Pasten.
\newblock Towards {H}ilbert’s tenth problem for rings of integers through
  {I}wasawa theory and heegner points.
\newblock {\em Mathematische Annalen}, 377(3):989--1013, 2020.

\bibitem[Gre01]{greenberg2001introduction}
Ralph Greenberg.
\newblock Introduction to {I}wasawa theory for elliptic curves.
\newblock {\em Arithmetic algebraic geometry (Park City, UT, 1999)},
  9:407--464, 2001.

\bibitem[GV00]{greenberg2000iwasawa}
Ralph Greenberg and Vinayak Vatsal.
\newblock On the {I}wasawa invariants of elliptic curves.
\newblock {\em Inventiones Mathematicae}, 142(1):17, 2000.

\bibitem[Har00]{harris2000correction}
Michael Harris.
\newblock Correction to p-adic representations arising from descent on abelian
  varieties.
\newblock {\em Compositio Mathematica}, 121(1):105--108, 2000.

\bibitem[HKR21]{hatley2021statistics}
Jeffrey Hatley, Debanjana Kundu, and Anwesh Ray.
\newblock Statistics for anticyclotomic iwasawa invariants of elliptic curves.
\newblock {\em arXiv preprint arXiv:2106.01517}, 2021.

\bibitem[HL21]{hung2021growth}
Pin-Chi Hung and Meng~Fai Lim.
\newblock On the growth of mordell--weil ranks in $ p $-adic lie extensions.
\newblock {\em Asian Journal of Mathematics}, 24(4):549--570, 2021.

\bibitem[HM99]{hachimori1999analogue}
Yoshitaka Hachimori and Kazuo Matsuno.
\newblock An analogue of {K}ida's formula for the selmer groups of elliptic
  curves.
\newblock {\em Journal of Algebraic Geometry}, 8(3):581--601, 1999.

\bibitem[KR21]{kunduraystats}
Debanjana Kundu and Anwesh Ray.
\newblock Statistics for {I}wasawa invariants of elliptic curves.
\newblock {\em Trans. Amer. Math. Soc.}, 374(11):7945--7965, 2021.

\bibitem[KR22]{kunduraykida}
Debanjana Kundu and Anwesh Ray.
\newblock {I}wasawa invariants for elliptic curves over
  {$\Bbb{Z}_p$}-extensions and {K}ida's formula.
\newblock {\em Forum Math.}, 34(4):945--967, 2022.

\bibitem[Mat70]{matiyasevich1970diophantineness}
Yuri~Vladimirovich Matiyasevich.
\newblock The {D}iophantineness of enumerable sets.
\newblock In {\em Doklady Akademii Nauk}, volume 191, pages 279--282. Russian
  Academy of Sciences, 1970.

\bibitem[MRL18]{mazur2018diophantine}
Barry Mazur, Karl Rubin, and Michael Larsen.
\newblock Diophantine stability.
\newblock {\em American Journal of Mathematics}, 140(3):571--616, 2018.

\bibitem[MRS24]{MRS2023}
Barry Mazur, Karl Rubin, and Alexandra Shlapentokh.
\newblock Existential definability and diophantine stability.
\newblock {\em J. Number Theory}, 254:1--64, 2024.

\bibitem[OP02]{OnoPapa}
Ken Ono and Matthew~A. Papanikolas.
\newblock Quadratic twists of modular forms and elliptic curves.
\newblock In {\em Number theory for the millennium, {III} ({U}rbana, {IL},
  2000)}, pages 73--85. A K Peters, Natick, MA, 2002.

\bibitem[Phe88]{pheidas1988hilbert}
Thanases Pheidas.
\newblock {H}ilbert’s tenth problem for a class of rings of algebraic
  integers.
\newblock {\em Proceedings of the American Mathematical Society},
  104(2):611--620, 1988.

\bibitem[Poo02]{poonen2002using}
Bjorn Poonen.
\newblock Using elliptic curves of rank one towards the undecidability of
  {H}ilbert’s tenth problem over rings of algebraic integers.
\newblock In {\em International Algorithmic Number Theory Symposium}, pages
  33--42. Springer, 2002.

\bibitem[Ray23]{Ray2023HTP}
Anwesh Ray.
\newblock Remarks on {H}ilbert's tenth problem and the {I}wasawa theory of
  elliptic curves.
\newblock {\em Bull. Aust. Math. Soc.}, 107(3):440--450, 2023.

\bibitem[Shl89]{shlapentokh1989extension}
Alexandra Shlapentokh.
\newblock Extension of {H}ilbert's tenth problem to some algebraic number
  fields.
\newblock {\em Communications on Pure and Applied Mathematics}, 42(7):939--962,
  1989.

\bibitem[Shl07]{shlapentokh2007hilbert}
Alexandra Shlapentokh.
\newblock {\em {H}ilbert's tenth problem: Diophantine classes and extensions to
  global fields}.
\newblock Number~7. Cambridge University Press, 2007.

\bibitem[Shl08]{shlapentokh2008elliptic}
Alexandra Shlapentokh.
\newblock Elliptic curves retaining their rank in finite extensions and
  {H}ilbert’s tenth problem for rings of algebraic numbers.
\newblock {\em Transactions of the American Mathematical Society},
  360(7):3541--3555, 2008.

\bibitem[Sil94]{silverman1994advanced}
Joseph~H Silverman.
\newblock {\em Advanced topics in the arithmetic of elliptic curves}, volume
  151.
\newblock Springer Science \& Business Media, 1994.

\bibitem[SS89]{shapiro1989diophantine}
Harold~N Shapiro and Alexandra Shlapentokh.
\newblock Diophantine relationships between algebraic number fields.
\newblock {\em Communications on Pure and Applied Mathematics},
  42(8):1113--1122, 1989.

\bibitem[Stu06]{sturm2006congruence}
Jacob Sturm.
\newblock On the congruence of modular forms.
\newblock In {\em Number Theory: A Seminar held at the Graduate School and
  University Center of the City University of New York 1984--85}, pages
  275--280. Springer, 2006.

\bibitem[Vid]{videla16decimo}
C~Videla.
\newblock Sobre el d{\'e}cimo problema de {H}ilbert.
\newblock {\em Atas da Xa Escola de Algebra, Vitoria, ES, Brasil. Colecao
  Atas}, 16:95--108.

\bibitem[vO15]{van15}
Jeanine van Order.
\newblock On the dihedral {E}uler characteristics of {S}elmer groups of abelian
  varieties.
\newblock {\em Arithmetic and Geometry}, 420:458, 2015.

\bibitem[Was97]{washington1997introduction}
Lawrence~C Washington.
\newblock {\em Introduction to cyclotomic fields}, volume~83.
\newblock Springer {S}cience \& {B}usiness Media, 1997.

\end{thebibliography}
\end{document}